\documentclass{amsart}    
\usepackage{graphicx}
\usepackage{hyperref}
\usepackage{amssymb}
\usepackage{setspace}
\title{A higher-order genus invariant and knot Floer homology}

\author{Peter D. Horn}
\address{Department of Mathematics\\Rice University\\6100 S. Main Street\\Houston, TX 77005 }
\email{phorn@rice.edu}
\urladdr{http://math.rice.edu/~phorn}

\newtheorem{thm}{Theorem}    % Standard theorem environment
\newtheorem{prop}{Proposition}

\theoremstyle{definition}
\newtheorem{defn}{Definition}    % Definition environment with 
\newtheorem{exam}{Example}
\newtheorem*{rem}{Remark}             % Unnumbered environment for remarks.
\newcommand{\inv}{^{-1}}
\newcommand{\p}{\pi_1}
\newcommand{\del}{\partial}
\renewcommand{\a}{\alpha}
\renewcommand{\b}{\beta}

\newcommand{\Gn}[1]{G^{(#1)}}

\newcommand{\Z}{{\mathbb{Z}}}
\newcommand{\Q}{{\mathbb{Q}}}

\newcommand{\N}{{\mathbb{N}}}
\newcommand{\dg}{{\mathrm{dg}}}

\begin{document}

\begin{abstract}
	It is known that knot Floer homology detects the genus and Alexander polynomial of a knot.  We investigate whether knot Floer homology of $K$ detects more structure of minimal genus Seifert surfaces for $K$.  We define an invariant of algebraically slice, genus one knots and provide examples to show that knot Floer homology does not detect this invariant.  Finally, we remark that certain metabelian $L^2$-signatures bound this invariant from below.
\end{abstract}

\maketitle

\section{Introduction}

	Knot Floer homology was defined by Peter Ozsv\'{a}th and Zolt\'{a}n Szab\'{o}~\cite{OS04b} and by Jacob Rasmussen~\cite{jR03}.  Knot Floer homology is a powerful knot invariant, and it detects such information as the Alexander polynomial~\cite{OS04b}\cite{jR03} and knot genus~\cite[Theorem 1.2]{OS04}.  Either of these invariants can be computed from a minimal genus Seifert surface.  We investigate whether knot Floer homology contains more information about any minimal genus Seifert surface.
	
	In~\cite{pH09}, the author defined a geometric invariant for knots in $S^{3}$ called the \textbf{first-order genus}.  Roughly, the first-order genus of $K$ is obtained by adding the individual genera (in $S^{3}-K$) of the curves in a symplectic basis on a minimal genus Seifert surface for $K$, and taking the minimum over all minimal genus Seifert surfaces.  The first-order genus of a knot is difficult to compute, as there are many symplectic bases for a given Seifert surface.  While difficult to compute in general, the first-order genus is a notion of higher-order genus defined for all knots.
	
	In this paper, we define a similar invariant, though it is only defined for algebraically slice, genus one knots.  We take a minimum over Seifert surfaces, but what we record is the genus (in $S^{3}-K$) of a basis curve which inherits the zero framing from the surface.  We will provide many examples and show that this invariant is not detected by knot Floer homology.
	
\begin{thm}\label{hfkfails}
	Knot Floer homology does not detect the knottedness of untwisted bands in a Seifert surface.
\end{thm}

	Cochran, Harvey and Leidy~\cite{CHL08} defined the first-order $L^2$-signatures of a knot.  By their definition, each algebraically slice, genus one knot has (at most) three first-order $L^2$-signatures.  We will discuss the relationship between our higher-order genus invariant and these first-order $L^2$-signatures.

\section{Motivation and definition}
	
	Let $\Sigma_{g}$ be a compact, oriented surface with one boundary component.  If $f:\Sigma_{g}\hookrightarrow S^{3}$ is an embedding with $K=f(\del \Sigma_{g})$, then some invariants of $K$ can be computed using this embedded surface $f(\Sigma_{g})$.  Such an surface is called a \textbf{Seifert surface} for $K$.  For example, any Seifert surface can be used to compute the knot's Alexander polynomial.  This polynomial is encoded in the knot Floer homology $\widehat{HFK}(K)$.  The smallest genus of such embedded surfaces with boundary $K$ is called the genus of $K$, $g(K)$, and this invariant is also detected by $\widehat{HFK}(K)$.  We na\"{i}vely ask whether $\widehat{HFK}(K)$ contains anymore information about the embedded surfaces with boundary $K$.  For example, we are interested in whether $\widehat{HFK}(K)$ contains information about the knottedness of certain simple closed curves on Seifert surfaces $f(\Sigma_{g})$ for $K$.  In this paper we will restrict our attention to genus one, algebraically slice knots.
	
	Our motivating example is the positively-clasped, untwisted Whitehead double of a knot $K$, denoted $D(K)$ and depicted in Figure~\ref{whd}.  There is an obvious genus one Seifert surface for $D(K)$.  In~\cite{pH09}, the author defined a knot invariant that measures the knottedness of the bands in Seifert surface for a knot.  For many knots $K$, this invariant applied to $D(K)$ `detects' the genus of $K$, i.e. $g_{1}(D(K))=1+g(K)$.  One may ask if $\widehat{HFK}(D(K))$ detects $g_{1}(D(K))\approx g(K)$, and by Hedden's formula~\cite[Theorem 1.2]{mH07}, the answer is `yes' in the sense that $\widehat{HFK}(D(K),1)$ has as a direct summand $\bigoplus_{j=-g(K)}^{g(K)} G_{j}(K)$, where the $G_{j}(K)$ are certain groups depending on $K$.  Due to computational difficulties, it is unknown whether $\widehat{HFK}(K)$ detects $g_{1}(K)$ in general.  We aim to define an invariant that is more computable than $g_{1}$ and which measures, more or less, the same thing.

\begin{figure}[ht]
	\begin{center}
	\begin{picture}(195, 80)(0,0)
		\put(93, 25){$\to$}
		\put(118, 28){\small $K$}
		\includegraphics{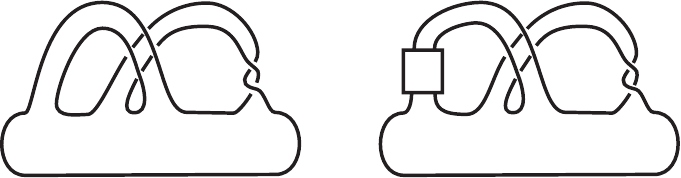}
	\end{picture}
	\caption{$D(K)$: the positively-clasped, untwisted Whitehead double of $K$}
	\label{whd}
	\end{center}
\end{figure}

\begin{defn}\label{met}
	Let $K$ be an algebraically slice knot in $S^{3}$ of genus one.  Let $\Sigma$ be any genus one Seifert surface for $K$.  Then $\Sigma$ has a \textbf{metabolizer} $\mathfrak{m}$, a rank one submodule of $H_{1}(\Sigma;\mathbb{Z})$ on which the Seifert form vanishes.  One can show that $\Sigma$ has exactly two metabolizers $\mathfrak{m}_{1}$ and $\mathfrak{m}_{2}$.  Let $[\a_{1}]$ and $[\a_{2}]\in H_{1}(\Sigma;\Z)$ be generators of $\mathfrak{m}_{1}$ and $\mathfrak{m}_{2}$, respectively.  By the classification of essential closed curves on a punctured torus~\cite{yM99}, each $[\a_{i}]$ determines a unique oriented knot in $\Sigma$; denote this knot by $\a_{i}$.  The knot $\a_{i}$ is called a \textbf{derivative of $K$ with respect to the metabolizer $\mathfrak{m}_{i}$}.
	
	To sum up, each genus one Seifert surface $\Sigma$ for an algebraically slice knot $K$ has exactly two (up to orientation) derivatives $\a_{1}$ and $\a_{2}$.  We denote this set of derivatives as $\partial\left(K,\Sigma\right)=\{\a_{1},\a_{2}\}$.
	
	Let $\mathcal{G}(K)$ denote the set of isotopy classes (in $S^{3}-K$) of genus one Seifert surfaces for $K$, and if $\a$ is a null-homologous knot in $S^{3}-K$, let $g^{K}(\a)$ denote the genus of $\a$ in $S^{3}-K$.  We define the \textbf{differential genus} of $K$ to be $$\dg(K)=\min_{\Sigma\in\mathcal{G}(K)}\left\{\max_{\partial(K,\Sigma)=\{\a_{1},\a_{2}\}}\left\{g^{K}(\a_{1}),g^{K}(\a_{2})\right\}\right\}$$
\end{defn}

\begin{rem}
		The differential genus measures the knottedness of self-linking zero curves on (genus one) Seifert surfaces for $K$.  One may define metabolizers and derivatives of algebraically slice knots of higher genus (see~\cite{CHL08}), but in the higher genus setting, a metabolizer may have infinitely many distinct derivatives.  One may try to generalize the definition of differential genus to higher genus algebraically slice knots; this may be taken up in a future paper.
\end{rem}

\section{Examples}

\begin{exam}
	Let $K$ be a knot that is non-trivial and not a cable.  By~\cite{wW73}, the untwisted Whitehead double of $K$, $D(K)$, has a unique minimal genus Seifert surface.  Each of the untwisted curves on this Seifert surface have the same knot type as $K$.  One can further argue that $\dg(D(K))=g(K)$.
\end{exam}

\begin{exam}
	Let $R=9_{46}$ as depicted in Figure~\ref{946}.  A symplectic basis of curves $\a$ and $\b$ have been drawn for the implied Seifert surface $\Sigma$.  One can check that $\a$ and $\b$ have self-linking zero, and so the two derivatives for $\Sigma$ are $\a$ and $\b$.  Each of $\a$ and $\b$ is unknotted, however $g^{R}(\a)=g^{R}(\b)=1$.  The knot Floer homology of $R$ is
	
	$$\begin{array}{ccc} 0&0&2\\0&5&0\\2&0&0\end{array}$$ where the $5$ appears in bigrading $(0,0)$.  The genus of $R$ is one, and since $\textrm{rank}\,\widehat{HFK}(R,1)=2<4$, we may apply Theorem 2.3 of~\cite{aJ08} to conclude that $\Sigma$ is the unique genus one Seifert surface for $R$ up to isotopy.  We conclude that $\dg(R)=1$.
	
		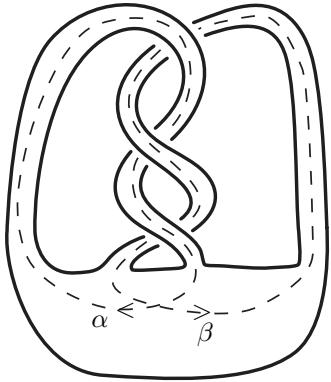
\begin{figure}[ht!]
		\begin{center}
			\begin{picture}(130, 140)(0,0)
			\includegraphics[scale=1.3]{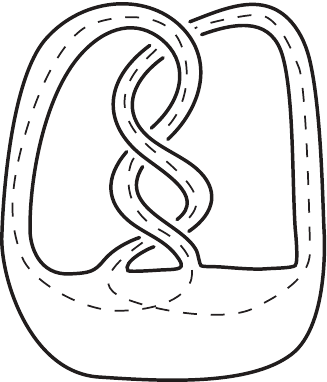}
			\put(-90, 20){$\a$}
			\put(-81, 24){$<$}
			\put(-50, 15){$\b$}
			\put(-52, 23){$>$}
			\end{picture}
			\caption{The $9_{46}$ knot}
			\label{946}
		\end{center}
	\end{figure}
	
\end{exam}

\begin{exam}\label{family}

	Now consider the knot $K_{n}$ in Figure~\ref{kn}, where $n\in\N$.  Observe that $K_{0}=9_{46}$.  A symplectic basis of curves $x$ and $y$ have been drawn for the implied Seifert surface $F$.  One can check the Seifert form of $F$ to be $$\left(\begin{array}{cc}3n&-2\\-1&0\end{array}\right)$$  The two curves of self-linking zero are $\a_{n}=x+ny$ and $\b_{n}=y$.  As in the calculation for $9_{46}$, one can check that $g^{K_{n}}(\b_{n})=1$.  The other curve $\a_{n}$ is more complicated; see Figure~\ref{alphan}.  The knot $\a_{n}$ can be represented by the braid on $n+1$ strands depicted in Figure~\ref{braid}.
	
\begin{figure}[ht!]
\begin{center}
\begin{picture}(280, 145)(0, -5)
{\begin{picture}(130, 140)(-75, 0)
			\includegraphics[scale=1.3]{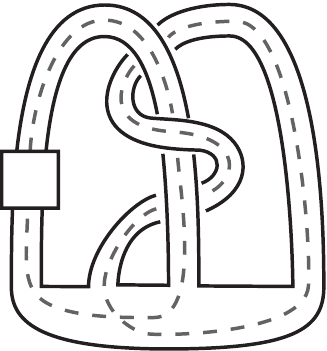}
			\put(-117, 62){$3n$}
			\put(-90, 5){$x$}
			\put(-92, 11){$<$}
			\put(-50, 12){$y$}
			\put(-45, 5){$>$}
			\end{picture}} 

\end{picture}
\end{center}
\caption{A diagram for $K_{n}$, and a basis for a Seifert surface} 
\label{kn} 
\end{figure}
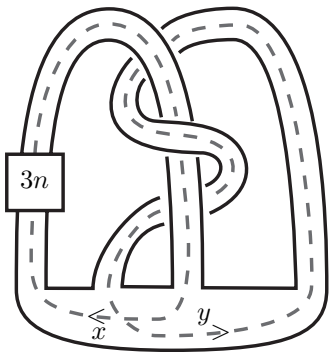 
	
	\begin{figure}[ht!]
		\begin{center}
	
		\begin{picture}(130, 130)(0,0)
			\put(36, 25){$\cdots$}
			\put(104, 25){$\cdots$}
			\put(37, 90){$\ddots$}
			\includegraphics[scale=1.3]{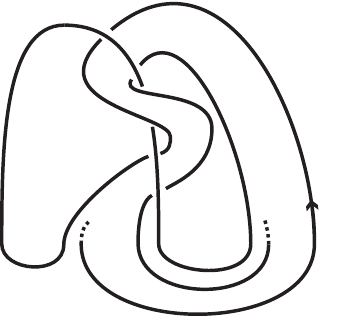}
			\end{picture}\label{alphan}
			\caption{A knot diagram of $\a_{n}$, a curve of self-linking zero}
		\end{center}
	\end{figure}
	
	\begin{figure}[ht!]
		\begin{center}
			\begin{picture}(90, 130)(0,0)
			\put(28, 5){$\cdots$}
			\put(33, 48){$\cdots$}
			\put(38, 91){$\cdots$}
			\includegraphics[scale=1.3]{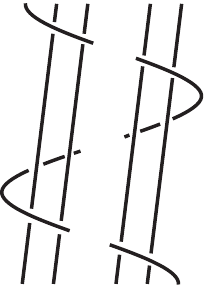}
			\end{picture}\label{braid}
			\caption{A braid representative of $\a_{n}$ using $n+1$ strands}
		\end{center}
	\end{figure}
	
	By~\cite[Corollary 4.1]{pC89}, the Seifert surface constructed by applying Seifert's algorithm to the braid diagram in Figure~\ref{braid} has minimal genus.  In particular $g(\a_{n})=n$, and hence $g^{K_{n}}(\a_{n})\geq n$.  We must argue that $\dg(K_{n})\geq n$.
	For a given $n$, one may easily construct a grid diagram for $K_{n}$.  For several small values of $n$, we used Marc Culler's Gridlink~\cite{mC} to compute the knot Floer homology of $K_{n}$.  We found that $\widehat{HFK}(K_{n})\cong\widehat{HFK}(9_{46})$ for these values of $n$ (although this family is defined for $n\in\N$, we verified the computation for $n=-1,\frac{-2}{3},\frac{-1}{3},0,\frac{1}{3},\frac{2}{3},1$).  A recent result of M. Hedden~\cite{mH08} implies that $\widehat{HFK}(K_{n})\cong\widehat{HFK}(9_{46})$ for all $n$.  By~\cite{aJ08}, $F$ is the unique genus one Seifert surface for $K_{n}$.  By previous arguments, we conclude that $\dg(K_{n})=g^{K_{n}}(\a_{n})\geq g(\a_{n})\geq n$.
\end{exam}

Theorem~\ref{hfkfails} follows.

\begin{rem}
	One can show by calculating the Alexander polynomials of the unknotted curves for $K_{1/3}$ that $\dg\left(K_{1/3}\right)\geq 2$.  Thus, $K_{1/3}$ is an explicit example of a knot with the same knot Floer homology as $9_{46}$ and distinct differential genus.  The knot $K_{1/3}$ is called $11_{n139}$ in the knot tables~\cite{CL09}.
\end{rem}

\begin{thm}
	There exists an infinite family of knots $K_{n}$ such that
	\begin{itemize}
		\item $\widehat{HFK}(K_{n})\cong\widehat{HFK}(K_{m})$ for all $m$ and $n$, and
		\item $\dg(K_{n})\neq \dg(K_{m})$ for $m\neq n$.
	\end{itemize}
\end{thm}
\begin{proof}
	The family is constructed by taking a subsequence of the knots $K_{n}$ from Example~\ref{family}.
\end{proof}

\section{First-order $L^2$-signatures and the differential genus}

Metabelian signatures of knots have been defined by Casson-Gordon, Letsche, Cochran-Orr-Teichner, Friedl, and Cochran-Harvey-Leidy~\cite{CG78}\cite{CG86}\cite{cL00}\cite{COT03}\cite{sF04}\cite{CHL08}.  We are interested in those of Cochran, Harvey, and Leidy because each genus one, algebraically slice knot has two ``first-order $L^2$-signatures.''  We now recall some of the background needed to define these signatures.

Suppose $K$ is an oriented knot in $S^{3}$, $M_{K}$ denote the closed, oriented $3$-manifold obtained by zero-surgery on $K$, and $G=\p(M_{K})$.  Let $\Gn{1}$ denote the commutator subgroup of $G$ and $\Gn{2}$ the commutator subgroup of $\Gn{1}$.  The \textbf{classical rational Alexander module} of $K$ is $$\mathcal{A}_{0}(K):= \frac{\Gn{1}}{\Gn{2}} \bigotimes_{\Z\left[t,t\inv\right]} \Q\left[t,t\inv\right]$$  Here $\Gn{1}/\Gn{2}$ is identified with the classical Alexander module $H_{1}\left(M_{K};\Z\left[t,t\inv\right]\right)$.  The \textbf{Blanchfield pairing} of $K$ $$\mathcal{B\ell}_{0}^{K}:\mathcal{A}_{0}(K)\times \mathcal{A}_{0}(K)\to\Q(t)/\Z\left[t,t\inv\right]$$ is defined by $$\mathcal{B\ell}_{0}^{K}(x,y)=\sum_{n\in\Z}\frac{(d\cdot yt^{n})t^{n}}{\Delta_{K}(t)}$$ where $\Delta_{K}(t)$ is the Alexander polynomial of $K$ and $d$ is a $2$-chain with $\del d=\Delta_{K}(t)\cdot x$.  We say a submodule $P\subset \mathcal{A}_{0}(K)$ is \textbf{Lagrangian} (respectively \textbf{isotropic}) if $P=P^{\perp}$ (respectively $P\subset P^{\perp}$) with respect to the Blanchfield pairing.  To a submodule $P\subset \mathcal{A}_{0}(K)$, we can associate a metabelian quotient $\phi_{P}:G\to G/\widetilde P$ by setting $\widetilde P=\ker\left(\Gn{1}\to\Gn{1}/\Gn{2}\to\mathcal{A}_{0}(K)\to\mathcal{A}_{0}(K)/P\right)$.  To this quotient we can associate a real number, called the Cheeger-Gromov von Neumann $\rho$-invariant, $\rho\left(M_{K},\phi_{P}\right)$~\cite{CG85} (see Chapter~\ref{hos} for a description).

\begin{defn}
	The \textbf{first-order $L^2$-signatures of a knot $K$} are the real numbers $\rho\left(M_{K},\phi_{P}\right)$ where $P$ is a Lagrangian submodule of $\mathcal{A}_{0}(K)$ with respect to $\mathcal{B\ell}_{0}^{K}$.
\end{defn}

\begin{rem}
	These are a subset of the metabelian $L^2$-signatures of Cochran, Harvey, and Leidy~\cite[Definition 4.1]{CHL08}, who allow for $P$ to be isotropic.
\end{rem}

Assume $K$ is a genus one, algebraically slice knot with a Seifert surface $\Sigma$.  The reader will recall that $H_{1}(\Sigma;\Z)$ generates $\mathcal{A}_{0}(K)$ as a $\Q\left[t,t\inv\right]$-module (one must pick a lift of $\Sigma$ to the infinite cyclic cover).  If $\Delta_{K}(t)=1$, then $\mathcal{A}_{0}(K)=0$ has no Lagrangian submodules.  On the other hand, if $\Delta_{K}(t)\neq 1$, then $\Delta_{K}(t)=f(t)f(t\inv)$ for some linear polynomial $f(t)$.  $\mathcal{A}_{0}(K)$ must be isomorphic to $\frac{\Q\left[t,t\inv\right]}{\langle f(t)f(t\inv)\rangle}$.  Thus, any proper submodule $P$ must be $$\frac{\Q\left[t,t\inv\right]}{\langle f(t)\rangle}\hspace{1cm}or\hspace{1cm}\frac{\Q\left[t,t\inv\right]}{\langle f(t\inv)\rangle}$$  Since the Blanchfield pairing is primitive, $\mathcal{A}_{0}(K)$ will have precisely two Lagrangians.
By the Definitions~\ref{met}, $K$ will have precisely two Lagrangians and hence two first-order $L^2$-signatures.

\begin{defn}\label{rep}
	Suppose $P\subset \mathcal{A}_{0}(K)$ is a Lagrangian.  The metabolizer $\mathfrak{m}$ \textbf{represents} $P$ if the image of $\mathfrak{m}$ under the map $$i_{\ast}\circ (\mathrm{id}\otimes 1):H_{1}(\Sigma;\Z)\hookrightarrow H_{1}(\Sigma;\Z)\otimes\Q\twoheadrightarrow \mathcal{A}_{0}(K)$$ spans $P$ as a $\Q$-vector space.  (To define $i_{\ast}$, it is necessary to choose a lift of $\Sigma$ to the infinite cyclic cover, but this definition is independent of the choice).
\end{defn}

\begin{prop}[Lemma 5.5 of~\cite{CHL08}]\label{CHLprop}
	Let $K$ be an algebraically slice knot and $P$ be a Lagrangian of $\mathcal{A}_{0}(K)$.  If $\Sigma$ is any Seifert surface for $K$, then some metabolizer of $H_{1}(\Sigma)$ represents $P$.
\end{prop}

\begin{prop}[Corollary 5.8 of~\cite{CHL08}]\label{keystep}
	Let $K$ be a genus one, algebraically slice knot.  Suppose $P$ is a Lagrangian for $K$, $\Sigma$ is a genus one Seifert surface for $K$, $\mathfrak{m}$ is the metabolizer of $\Sigma$ representing $P$, and $J$ is the derivative with respect to $\mathfrak{m}$. Then the first-order $L^2$-signature of $K$ with respect to $P$ is equal to $\rho_{0}(J)=\int_{S^{1}}\sigma_{\omega}(J)\, d\omega$, the integral of the Levine-Tristram signature function.
\end{prop}

Determining $\dg(K)$ involves computing the genus of two curves from each genus one Seifert surface, of which there may be many.  Examples of knots that have an arbitrary number of non-isotopic Seifert surfaces are known~\cite[p. 47]{sS91}.  Yet we have the following remarkable fact: if just one of the first-order $L^2$-signatures is large, then the differential genus must be large.

\begin{prop}\label{lowerbound}
	Let $K$ be a genus one, algebraically slice knot with non-trivial Alexander polynomial.  Let $\rho_{1}$ and $\rho_{2}$ denote the first-order $L^2$-signatures of $K$ with respect two the two Lagrangians $P_{1}$ and $P_{2}$.  Then $2\,\dg(K)\geq \max\{|\rho_{1}|,|\rho_{2}|\}$.
\end{prop}

\begin{proof}
	Let $\Sigma$ be the Seifert surface where the minimum is attained.  For either derivative $J_{i}\subset \Sigma$, where $J_{i}$ represents the Lagrangian $P_{i}$, we have $$2\,\dg(K)\geq 2\,g^{K}(J_{i})\geq 2\,g(J_{i})\geq\left|\int_{S^{1}}\sigma_{\omega}(J_{i})\,d \omega\right|\geq|\rho_{i}|$$
\end{proof}

	\bibliographystyle{amsalpha}
	\bibliography{hfkbib}
\end{document}